\documentclass[12pt]{amsart}

\usepackage[mathscr]{eucal}
\usepackage{amscd}
\usepackage{amsfonts}
\usepackage{amsmath, amsthm, amssymb}
\usepackage{latexsym}
\usepackage[dvips]{graphics}

\theoremstyle{plain}
\newtheorem{lem}{Lemma}[section]
\newtheorem{prop}[lem]{Proposition}

\newtheorem{thm}[lem]{Theorem}

\theoremstyle{remark}

\theoremstyle{definition}
\newtheorem{exm}[lem]{Example}
\newtheorem{defn}[lem]{Definition}

\newcommand{\Fun}{\operatorname{Fun}\nolimits}

\newcommand{\id}{\operatorname{id}\nolimits} 
 
\newcommand{\Sub}{\operatorname{Sub}\nolimits}

\newcommand{\op}{\mathrm{op}}

\newcommand{\inj}{\mathrm{inj}}

\newcommand{\os}{\mathrm{os}}
\newcommand{\sur}{\mathrm{sur}}
\newcommand{\Set}{\mathrm{Set}}

\newcommand{\lto}[1][{}]{\stackrel{#1}{\longrightarrow}} 
\renewcommand{\to}[1][{}]{\stackrel{#1}{\rightarrow}} 
\newcommand{\xto}{\xrightarrow}

\def\a{\alpha}

\def\p{\phi}

\def\s{\sigma}
\def\t{\tau}

\def\m{\mu}

\def\la{\lambda}

\def\Ga{\Gamma}

\def\A{{\mathcal A}}

\def\C{{\mathcal C}}
\def\D{{\mathcal D}}

\def\P{{\mathcal P}}

\def\T{{\mathcal T}}

\def\bbN{\mathbb N}

\numberwithin{equation}{section}

\title{The artinian conjecture \\ 
(following Djament, Putman, Sam, and Snowden) }

\author{Henning Krause} 

\address{
\begin{flushleft}
        \hspace{0.3cm}  Fakult\"at f\"ur Mathematik \\
         \hspace{0.3cm}  Universit\"at Bielefeld\\
         \hspace{0.3cm}  33501 Bielefeld, Germany\\
\end{flushleft}
}
\email{hkrause@math.uni-bielefeld.de} 

\thanks{The paper is in a final form and no version of it will 
be submitted for publication elsewhere.}

\begin{document}

\maketitle


\begin{abstract}
 This note provides a self-contained exposition of the proof
  of the artinian conjecture, following closely Djament's Bourbaki
  lecture.  The original proof is due to Putman, Sam, and Snowden.



\end{abstract}


\section{Introduction}

This note provides a complete proof of the celebrated artinian
conjecture. The proof is due to Putman, Sam, and Snowden
\cite{PS,SS}. Here, we follow closely the elegant exposition of
Djament in \cite{Dj1}. For the origin of the conjecture and its
consequences, we refer to those papers and Djament's Bourbaki lecture
\cite{Dj2}. In addition, the expository articles by Kuhn, Powell and
Schwartz in \cite{KR2000} are recommended.

There are two main result. Fix a locally noetherian Grothendieck
abelian category $\A$, for instance, the category of modules over a
noetherian ring.

\begin{thm}
  Let $A$ be a ring whose underlying set is finite. For the category
  $\P(A)$ of free $A$-modules of finite rank,
  the functor category $\Fun(\P(A)^\op,\A)$ is locally noetherian.
\end{thm}

This result amounts to the assertion of the artinian conjecture when
$A$ is a finite field and $\A$ is the category of $A$-modules. 

The first theorem is a direct consequence of the following.

\begin{thm}
For the category $\Ga$ of finite sets, the functor category
  $\Fun(\Ga^\op,\A)$ is locally noetherian.
\end{thm}

The basic idea for the proof is to formulate finiteness conditions on
an essentially small category $\C$ such that $\Fun(\C^\op,\A)$ is
locally noetherian. This leads to the notion of a Gr\"obner
category. Such finiteness conditions have a `direction'. For that
reason we consider contravariant functors $\C\to\A$, because then the
direction is preserved (via Yoneda's lemma) when one passes from $\C$
to $\Fun(\C^\op,\A)$. 

\section{Noetherian posets}

Let $\C$ be a poset. A subset $\D\subseteq\C$ is a \emph{sieve} if
the conditions $x\le y$ in $\C$ and $y\in\D$ imply $x\in\D$. The
sieves in $\C$ are partially ordered by inclusion. 

\begin{defn}
A poset $\C$ is called
\begin{enumerate}
\item \emph{noetherian}  if every ascending chain of
  elements in $\C$ stabilises, and
\item  \emph{strongly noetherian}  if every ascending chain of
  sieves in $\C$ stabilises.
\end{enumerate}
\end{defn}

For a poset $\C$ and $x\in\C$, set $\C(x)=\{t\in\C\mid t\le x\}$. The
assignment $x\mapsto\C(x)$ yields an embedding of $\C$ into the poset
of sieves in $\C$.

\begin{lem}\label{le:stronglynoeth}
 For a poset $\C$ the following are equivalent:
\begin{enumerate}
\item The poset $\C$ is strongly noetherian.
\item For every infinite sequence $(x_i)_{i\in\bbN}$
  of elements in $\C$ there exists $i\in\bbN$ such that $x_j\le x_i$
  for infinitely many $j\in\bbN$. 
\item For every infinite sequence $(x_i)_{i\in\bbN}$ of elements in
  $\C$ there is a map $\a\colon\bbN\to\bbN$ such that $i<j$ implies
  $\a(i)<\a(j)$ and $x_{\a(j)}\le x_{\a(i)}$.
\item For every infinite sequence $(x_i)_{i\in\bbN}$ of elements in
  $\C$ there are $i<j$ in $\bbN$ such that $x_j\le x_i$.
\end{enumerate}
\end{lem}
\begin{proof} 
  (1) $\Rightarrow$ (2): Suppose that $\C$ is strongly noetherian and
  let $(x_i)_{i\in\bbN}$ be elements in $\C$.  For $n\in\bbN$ set
  $\C_n=\bigcup_{i\le n}\C(x_i)$. The chain $(\C_n)_{n\in\bbN}$
  stabilises, say $\C_n=\C_N$ for all $n\ge N$. Thus there exists
  $i\le N$ such that $x_j\le x_i$ for infinitely many $i\in\bbN$.

  (2) $\Rightarrow$ (3): Define $\a\colon\bbN\to\bbN$ recursively by
  taking for $\a(0)$ the smallest $i\in\bbN$ such that $x_j\le x_i$
  for infinitely many $j\in\bbN$. For $n> 0$ set \[\a(n)=\min\{i >
  \a(n-1)\mid x_j\le x_i\le x_{\a(n-1)} \textrm{ for infinitely many
  }j\in\bbN\}.\]
 
 (3) $\Rightarrow$ (4): Clear.

 (4) $\Rightarrow$ (1): Suppose there is a properly ascending chain
 $(\C_n)_{n\in\bbN}$ of sieves in $\C$. Choose
 $x_n\in\C_{n+1}\setminus \C_{n}$ for each $n\in\bbN$. There are $i<j$
 in $\bbN$ such that $x_j\le x_i$.  This implies
 $x_j\in\C_{i+1}\subseteq\C_{j}$ which is a contradiction.
\end{proof}

\section{Functor categories}

Let $\C$ be an essentially small category and $\A$ a Grothendieck
abelian category. We denote by $\Fun(\C^\op,\A)$ the category of functors
$\C^\op\to\A$. The morphisms between two functors are the natural
transformations. Note that $\Fun(\C^\op,\A)$ is a Grothendieck abelian
category.

Given an object $x\in\C$, the evaluation functor
\[\Fun(\C^\op,\A)\lto\A,\quad F\mapsto  F(x)\]
admits a left adjoint
\[\A\lto \Fun(\C^\op,\A),\quad M\mapsto  M[\C(-,x)]\]
where for any set $X$ we denote by $M[X]$ a coproduct of copies of $M$
indexed by the elements of $X$. Thus we have a natural isomorphism
\begin{equation}\label{eq:adj}
\Fun(\C^\op,\A)(M[\C(-,x)],F)\cong\A(M,F(x)).
\end{equation}

\begin{lem}\label{le:gen}
If $(M_i)_{i\in I}$ is a set of generators of $\A$, then the functors
$M_i[\C(-,x)]$ with $i\in I$ and $x\in\C$ generate $\Fun(\C^\op,\A)$.
\end{lem}
\begin{proof}
Use the adjointness isomorphism \eqref{eq:adj}.
\end{proof}

A Grothendieck abelian category $\A$ is \emph{locally noetherian} if
$\A$ has a generating set of noetherian objects. In that case an
object $M\in\A$ is noetherian iff $M$ is \emph{finitely presented}
(that is, the representable functor $\A(M,-)$ preserves filtered
colimits); see \cite[Chap.~V]{St1975} for details.

\begin{lem}\label{le:noeth}
Let $\A$ be locally noetherian. Then $\Fun(\C^\op,\A)$ is locally
noetherian iff $M[\C(-,x)]$ is noetherian for every noetherian
$M\in\A$ and $x\in\C$.
\end{lem}
\begin{proof}
  First observe that $M[\C(-,x)]$ is finitely presented if $M$ is
  finitely presented. This follows from the isomorphism \eqref{eq:adj}
  since evaluation at $x\in\C$ preserves colimits. Now the assertion
  of the lemma is an immediate consequence of Lemma~\ref{le:gen}.
\end{proof}

\section{Noetherian functors}

Let $\C$ be an essentially small category and fix an object
$x\in\C$. Set
\[\C(x)=\bigsqcup_{t\in\C}\C(t,x).\]
Given $f,g\in\C(x)$, let $\langle f\rangle$ denote the set of morphisms
in $\C(x)$ that factor through $f$, and set $f\le_x g$
if $\langle f\rangle\subseteq\langle g\rangle$. We identify $f$ and
$g$  when $\langle f\rangle=\langle g\rangle$. This yields a
poset which we denote by $\bar\C(x)$.

A functor is  \emph{noetherian} if every ascending chain of
subfunctors stabilises.

\begin{lem}\label{le:repnoeth}
  The functor $\C(-,x)\colon\C^\op\to\Set$ is noetherian iff the poset
  $\bar\C(x)$ is strongly noetherian.
\end{lem}
\begin{proof}
Sending  $F\subseteq\C(-,x)$ to $\bigcup_{t\in\C}F(t)$
induces an inclusion preserving bijection between the subfunctors of
$\C(-,x)$ and the sieves in $\bar\C(x)$.
\end{proof}

For a poset $\T$ let $\Set\wr\T$ denote the category consisting of
pairs $(X,\xi)$ such that $X$ is a set and $\xi\colon X\to\T$ is a
map. A morphism $(X,\xi)\to (X',\xi')$ is a map $f\colon X\to X'$ such
that $\xi(a)\le\xi' f(a)$ for all $a\in X$.

A functor $\C^\op\to\Set\wr\T$ is given by a pair $(F,\p)$ consisting
of a functor $F\colon\C^\op\to\Set$ and a map
$\p\colon\bigsqcup_{t\in\C}F(t)\to\T$ such that $\p(a)\le \p(F(f)(a))$
for every $a\in F(t)$ and $f\colon t'\to t$ in $\C$.
 
\begin{lem}\label{le:noeth3}
Let $\T$ be a  noetherian poset. If $\C(-,x)$ is noetherian,
then any functor $(\C(-,x),\p)\colon\C^\op\to\Set\wr\T$ is noetherian.
\end{lem}
\begin{proof}
  Let $(F_n,\p_n)_{n\in\bbN}$ be a strictly ascending chain of
  subfunctors of $(F,\p)$. The chain $(F_n)_{n \in\bbN}$ stabilises
  since $\C(-,x)$ is noetherian. Thus we may assume that $F_n=F$ for
  all $n\in\bbN$, and we find $f_n\in \bigsqcup_{t\in\C}F(t)$ such
  that $\p_{n}(f_{n})<\p_{n+1}(f_{n})$. The poset $\bar\C(x)$ is
  strongly noetherian by Lemma~\ref{le:repnoeth}. It follows from
  Lemma~\ref{le:stronglynoeth} that there is a map
  $\a\colon\bbN\to\bbN$ such that $i<j$ implies $\a(i)<\a(j)$ and
  $f_{\a(j)}\le_x f_{\a(i)}$. Thus
  \[\p_{\a(n)}(f_{\a(n)})<\p_{\a(n)+1}(f_{\a(n)})\le
  \p_{\a(n+1)}(f_{\a(n)})\le \p_{\a(n+1)}(f_{\a(n+1)}).\] This yields
  a strictly ascending chain in $\T$, contradicting the assumption on
  $\T$.
\end{proof}

\begin{defn}
  A partial order $\le$ on $\C(x)$ is \emph{admissible} if the
  following holds:
\begin{enumerate}
\item The  order $\le$ restricted to $\C(t,x)$ is total and noetherian for every $t\in\C$. 
\item For $f,f'\in\C(t,x)$ and $e\in\C(s,t)$, the condition $f<f'$
  implies $fe<f'e$.
\end{enumerate}
\end{defn}

Fix an admissible partial order $\le$ on $\C(x)$ and an object $M$ in
a Grothendieck abelian category $\A$. Let $\Sub(M)$ denote the
poset of subobjects of $M$ and consider the functor
\[\C(-,x)\wr M\colon\C^\op\lto \Set\wr\Sub(M),\quad t\mapsto 
\big(\C(t,x),(M)_{f\in\C(t,x)}\big).\] For a subfunctor $F\subseteq
M[\C(-,x)]$ define a subfunctor $\tilde F\subseteq \C(-,x)\wr M$ as
follows:
\[\tilde F\colon\C^\op\lto\Set\wr\Sub(M), \quad t\mapsto
\Big(\C(t,x),\big(\pi_f(M[\C(t,x)_f]\cap F(t))\big)_{f\in\C(t,x)}
\Big)\] where $\C(t,x)_f=\{g\in\C(t,x)\mid f\le g\}$ and $\pi_f\colon
M[\C(t,x)_f]\to M$ is the projection onto the factor corresponding to
$f$. For a morphism $e\colon t'\to t$ in $\C$, the morphism $\tilde
F(e)$ is induced by precomposition with $e$. Note that 
\[\pi_f(M[\C(t,x)_f]\cap F(t))\subseteq \pi_{fe}(M[\C(t',x)_{fe}]\cap F(t'))\]
since $\le$ is compatible with the composition in $\C$.

\begin{lem}\label{le:tilde}
  Suppose there is an admissible partial order on $\C(x)$.  Then the
  assignment which sends a subfunctor $F\subseteq M[\C(-,x)]$ to
  $\tilde F$ preserves proper inclusions.  Therefore $M[\C(-,x)]$ is
  noetherian provided that $\C(-,x)\wr M$ is noetherian.
\end{lem}
\begin{proof}
  Let $F\subseteq G \subseteq M[\C(-,x)]$. Then $\tilde
  F\subseteq\tilde G$.  Now suppose that $F\neq G$. Thus there exists
  $t\in \C$ such that $F(t)\neq G(t)$. We have
  $\C(t,x)=\bigcup_{f\in\C(t,x)}\C(t,x)_f$, and this union is directed
  since $\le$ is total. Thus
  \[F(t)=\sum_{f\in\C(t,x)}\big(M[\C(t,x)_f]\cap F(t)\big)\] since
  filtered colimits in $\A$ are exact. This yields $f$ such that
  \[M[\C(t,x)_f]\cap F(t)\neq M[\C(t,x)_f]\cap G(t).\] Choose
  $f\in\C(t,x)$ maximal with respect to this property, using that $\le$ is
  noetherian.  Now observe that the projection
$\pi_f$ induces  an exact sequence
\[0\lto \sum_{f<g}\big(M[\C(t,x)_g]\cap F(t)\big)\lto F(t)\lto
\pi_f\big(M[\C(t,x)_f]\cap F(t)\big)\lto 0\] since the kernel of
$\pi_f$ equals the directed union $\sum_{f<g}M[\C(t,x)_g]$. For the
directedness one uses again that $\le$ is total.
Thus \[\pi_f\big(M[\C(t,x)_f]\cap F(t)\big)\neq
\pi_f\big(M[\C(t,x)_f]\cap G(t)\big)\] and therefore $\tilde
F\neq\tilde G$.
\end{proof}

\begin{prop}\label{pr:groebner0}
  Let $x\in\C$. Suppose that $\C(-,x)$ is noetherian and that $\C(x)$
  has an admissible partial order. If $M\in\A$ is noetherian, then $M[\C(-,x)]$ is noetherian.
\end{prop}
\begin{proof}
  Combine  Lemmas~\ref{le:noeth3} and \ref{le:tilde}.
\end{proof}

\section{Gr\"obner categories}

\begin{defn}
An essentially small category $\C$  is a \emph{Gr\"obner category} if the
following holds:
\begin{enumerate}
\item The functor $\C(-,x)$ is noetherian for every $x\in\C$.
\item There is an admissible partial order on $\C(x)$ for  every $x\in\C$.
\end{enumerate}
\end{defn}

\begin{thm}\label{th:groebner}
Let $\C$ be a Gr\"obner category and $\A$ a  Grothendieck abelian
category. If $\A$ is locally noetherian, then $\Fun(\C^\op,\A)$ is
locally noetherian.
\end{thm}
\begin{proof}
Combine Lemma~\ref{le:gen} and Proposition~\ref{pr:groebner0}.
\end{proof}

\begin{exm}
(1) A strongly noetherian poset (viewed as a category) is a Gr\"obner category. 

(2) The additive monoid $\bbN$ of natural numbers (viewed
as a category with a single object) is a Gr\"obner category.
Let $\A$ be the module
category of a noetherian ring $A$. Then  $\Fun(\bbN^\op,\A)$ equals the
module category of the polynomial ring in one variable  over $A$. Thus
Theorem~\ref{th:groebner} generalises Hilbert's Basis Theorem.
\end{exm}

\section{Base change}

Given functors $F,G\colon\C^\op\to\Set$, we write $F\leadsto G$
if there is a finite chain
\[F=F_0\twoheadrightarrow F_1\leftarrowtail F_2 \twoheadrightarrow
\cdots \twoheadrightarrow F_{n-1}\leftarrowtail F_{n} =G\]
of epimorphisms  and monomorphisms of functors $\C^\op\to\Set$.

\begin{defn}
  A functor $\p\colon\C\to\D$ is \emph{contravariantly
    finite}\footnote{The terminology follows that introduced by
    Auslander and Smal{\o} \cite{AS1980} for an inclusion functor.} if
  the following holds:
\begin{enumerate}
\item Every object $y\in\D$ is isomorphic to $\p(x)$ for some $x\in\C$.
\item For every object $y\in\D$ there are objects $x_1,\ldots,x_n$ in
  $\C$ such that \[\bigsqcup_{i=1}^n\C(-,x_i)\leadsto \D(\p-,y).\]
\end{enumerate}
The functor $\p$ is \emph{covariantly finite} if
$\p^\op\colon\C^\op\to\D^\op$ is contravariantly finite.
\end{defn}

Note that a composite of contravariantly finite functors is 
contravariantly finite.

\begin{lem}\label{le:contra}
Let  $f\colon\C\to\D$ be a contravariantly finite functor and $\A$
a Grothendieck abelian category. Fix $M\in\A$ and suppose that
$M[\C(-,x)]$ is noetherian for all $x\in\C$. Then $M[\D(-,y)]$ is
noetherian for all $y\in\D$.
\end{lem}
\begin{proof}
A finite chain
\[\bigsqcup_{i=1}^n \C(-,x_i)=F_0\twoheadrightarrow F_1\leftarrowtail F_2 \twoheadrightarrow
\cdots \twoheadrightarrow F_{n-1}\leftarrowtail F_{n} =\D(\p-,y)\] of
epimorphisms and monomorphisms induces
a chain
\[\coprod_{i=1}^n M[\C(-,x_i)]=\bar F_0\twoheadrightarrow \bar
F_1\leftarrowtail \bar F_2 \twoheadrightarrow \cdots
\twoheadrightarrow \bar F_{n-1}\leftarrowtail \bar F_{n}
=M[\D(\p-,y)]\] of epimorphisms and monomorphisms in
$\Fun(\C^\op,\A)$. Thus $M[\D(\p-,y)]$ is noetherian. It follows that
$M[\D(-,y)]$ is noetherian, since precomposition with $\p$ yields a
faithful and exact functor $\Fun(\D^\op,\A)\to \Fun(\C^\op,\A)$.
\end{proof}

\begin{prop}\label{pr:contra}
Let  $f\colon\C\to\D$ be a contravariantly finite functor and $\A$
a locally noetherian Grothendieck abelian category. If the category
$\Fun(\C^\op,\A)$ is locally noetherian, then $\Fun(\D^\op,\A)$ is locally noetherian.
\end{prop}
\begin{proof}
Combine Lemmas~\ref{le:noeth} and \ref{le:contra}.
\end{proof}

\section{Categories of finite sets}

Let $\Ga$ denote the category of finite sets (a skeleton is given by
the sets $\mathbf n= \{1,2,\ldots,n\}$). The subcategory of finite
sets with surjective morphisms is denoted by $\Ga_\sur$. A surjection
$f\colon\mathbf m\to\mathbf n$ is \emph{ordered} if $i<j$ implies
$\min f^{-1}(i)<\min f^{-1}(j)$. We write $\Ga_\os$ for the
subcategory of finite sets whose morphisms are ordered
surjections. Given a surjection $f\colon\mathbf m\to\mathbf n$, let
$f^!\colon\mathbf n\to\mathbf m$ denote the map given by $f^!(i)=\min
f^{-1}(i)$. Note that $ff^!=\id$, and $gf=f^!g^!$ provided that $f$
and $g$ are ordered surjections.

\begin{lem}\label{le:contra1}
\begin{enumerate}
\item The inclusion $\Ga_\sur\to\Ga$ is contravariantly finite. 
\item The inclusion $\Ga_\os\to\Ga_\sur$ is contravariantly finite. 
\end{enumerate}
\end{lem}
\begin{proof}
(1) For each integer $n\ge 0$ there is an isomorphism 
\[\bigsqcup_{\mathbf m\hookrightarrow\mathbf n}\Ga_\sur(-,\mathbf
m)\xto{\sim}\Ga(-,\mathbf n) \]
which is induced by the injective maps $\mathbf m\to\mathbf n$.

(2) For each integer $n\ge 0$ there is an isomorphism
\[\Ga_\os(-,\mathbf n) \times\mathfrak S_n\xto{\sim}\Ga_\sur(-,\mathbf n)\]
which sends a pair $(f,\s)$ to $\s f$. The inverse sends a surjective
map $g\colon\mathbf m\to\mathbf n$ to $(\t^{-1}g,\t)$ where
$\t\in\mathfrak S_n$ is the unique permutation such that $g^!\t$ is
increasing.
\end{proof}

Fix an integer $n\ge 0$. Given $f,g\in\Ga(\mathbf n)$ we set $f\le g$
if there exists an ordered surjection $h$ such that $f=gh$.

\begin{lem}\label{le:higman}
 The poset $(\Ga(\mathbf n),\le)$ is
  strongly noetherian.
\end{lem}
\begin{proof}
  We fix some notation for each $f\in\Ga(\mathbf m,\mathbf n)$. Set
  $\la (f)=m$. If $f$ is not injective, set
\[\m(f)=m-\max\{i\in\mathbf m\mid \text{there exists $j<i$ such that
}f(i)=f(j)\}\] and $\pi(f)=f(m-\m(f))$.  Define $\tilde
f\in\Ga(\mathbf{m -1},\mathbf n)$ by setting $\tilde f(i)=f(i)$ for
$i< m-\m(f)$ and $\tilde f(i)=f(i+1)$ otherwise.

Note that $f\le\tilde f$. Moreover, $\m(f)=\m(g)$, $\pi(f)=\pi(g)$,
and $\tilde f\le\tilde g$ imply $f\le g$.

Suppose that $(\Ga(\mathbf n),\le)$ is not strongly noetherian. Then
there exists an infinite sequence $(f_r)_{r\in\bbN}$ in $\Ga(\mathbf
n)$ such that $i<j$ implies $f_j\not\le f_i$; see
Lemma~\ref{le:stronglynoeth}. Call such a sequence \emph{bad}. Choose
the sequence \emph{minimal} in the sense that $\la (f_i)$ is minimal
for all bad sequences $(g_r) _{r\in\bbN}$ with $g_j=f_j$ for all
$j<i$. There is an infinite subsequence $(f_{\a(r)})_{r\in\bbN}$
(given by some increasing map $\a\colon\bbN\to\bbN$) such that $\m$
and $\pi$ agree on all $f_{\a(r)}$, since the values of $\m$ and $\pi$
are bounded by $n$. Now consider the sequence
$f_0,f_1,\ldots,f_{\a(0)-1},\tilde f_{\a(0)},\tilde f_{\a(1)},\ldots$
and denote this by $(g_r) _{r\in\bbN}$. This sequence is not bad,
since $(f_r)_{r\in\bbN}$ is minimal. Thus there are $i<j$ in $\bbN$
with $g_j\le g_i$. Clearly, $j<\a(0)$ is impossible. If $i<\a(0)$,
then
\[ f_{\a(j-\a(0))}\le \tilde f_{\a(j-\a(0))}=g_j\le g_i=f_i,\] which
is a contradiction, since $i<\a(0)\le\a(j-\a(0))$. If $i\ge \a(0)$,
then $f_{\a(j-\a(0))}\le f_{\a(i-\a(0))}$; this is a contradiction
again. Thus $(\Ga(\mathbf n),\le)$ is strongly noetherian.
\end{proof}

\begin{prop}\label{pr:groebner}
The category $\Ga_\os$ is a Gr\"obner category.
\end{prop}
\begin{proof}
  Fix an integer $n\ge 0$. The poset $\bar \Ga_\os(\mathbf n)$ is
  strongly noetherian by Lemma~\ref{le:higman}, and it follows from
  Lemma~\ref{le:repnoeth} that the functor $\Ga_\os(-,\mathbf n)$ is
  noetherian.

The admissible partial order on $\Ga_\os(\mathbf n)$ is given by the
lexicographic order. Thus for $f,g\in\Ga_\os(\mathbf m,\mathbf n)$, we
have $f<g$ if there exists $j\in\mathbf m$ with $f(j)<g(j)$ and
$f(i)=g(i)$ for all $i<j$.
\end{proof}

\begin{thm}\label{th:sets}
  Let $\A$ be a locally noetherian Grothendieck abelian category. Then
  the category $\Fun(\Ga^\op,\A)$ is locally noetherian.
\end{thm}
\begin{proof}
  The category $\Ga_\os$ is a Gr\"obner category by
  Proposition~\ref{pr:groebner}. It follows from Theorem~\ref{th:groebner}
  that $\Fun((\Ga_\os)^\op,\A)$ is locally noetherian.  The inclusion
  $\Ga_\os\to\Ga$ is contravariantly finite by
  Lemma~\ref{le:contra1}. Thus $\Fun(\Ga^\op,\A)$ is locally
  noetherian by Proposition~\ref{pr:contra}.
\end{proof}

\section{The artinian conjecture}

Let $A$ be a ring. We denote by $\P(A)$ the category of free
$A$-modules of finite rank. If $A$ is finite, then the functor
$\Ga\to\P(A)$ sending $X$ to $A[X]$ is a left adjoint of the forgetful
functor $\P(A)\to\Ga$.

\begin{lem}\label{le:contra2}
Let $A$ be finite. Then the  functor $\Ga\to\P(A)$ is contravariantly
finite.
\end{lem}
\begin{proof}
The assertion follows from the adjointness isomorphism
\[\P(A)(A[X],P)\cong\Ga(X,P).\qedhere\]
\end{proof}

\begin{thm}
Let $A$ be a finite ring and $\A$ a locally noetherian Grothendieck
abelian category. Then the category $\Fun(\P(A)^\op,\A)$ is locally noetherian.
\end{thm}
\begin{proof}
  Combine Theorem~\ref{th:sets} with Lemma~\ref{le:contra2} and
  Proposition~\ref{pr:contra}.
\end{proof}

\section{FI-modules}

The proof of the artinian conjecture yields an alternative proof of
the following result due to Church, Ellenberg, Farb, and Nagpal.

Let $\Ga_\inj$ denote the category whose objects are finite sets and
whose morphisms are injective maps.

\begin{thm}[{\cite[Theorem~A]{CEFN}}]
 Let $\A$ be a locally noetherian Grothendieck abelian category. Then the
  category $\Fun(\Ga_\inj,\A)$ is locally noetherian.
\end{thm}
\begin{proof} The following argument has been suggested by Kai-Uwe
  Bux.  Consider the functor $\p\colon \Ga_\os\to(\Ga_\inj)^\op$ which
  is the identity on objects and takes a map $f\colon\mathbf
  m\to\mathbf n$ to $f^!\colon\mathbf n\to\mathbf m$ given by
  $f^!(i)=\min f^{-1}(i)$. This functor is contravariantly finite,
  since for each integer $n\ge 0$ the morphism
\[\Ga_\os(-,\mathbf n)\times\mathfrak S_n \lto\Ga_\inj(\mathbf n,\p-)\]
which sends a pair $(f,\s)$ to $f^!\s$ is an epimorphism.

It follows from Proposition~\ref{pr:contra} that the category
$\Fun(\Ga_\inj,\A)$ is locally noetherian, since
$\Fun((\Ga_\os)^\op,\A)$ is locally noetherian by
Proposition~\ref{pr:groebner} and Theorem~\ref{th:groebner}.
\end{proof}

\section*{Note added in proof}

After completing this paper I found that Theorem~\ref{th:groebner} is
precisely the statement of Theorem~3.1 in [G. Richter, Noetherian
semigroup rings with several objects, in {\it Group and semigroup
  rings (Johannesburg, 1985)}, 231--246, North-Holland Math. Stud.,
126, North-Holland, Amsterdam, 1986].


 
\ifx\undefined\bysame 
\newcommand{\bysame}{\leavevmode\hbox to3em{\hrulefill}\,} 
\fi 


\end{document}